\documentclass[12pt]{amsart}
\usepackage{amssymb,latexsym}
\usepackage{amsfonts}
\usepackage{amsmath}
\usepackage[colorlinks,linkcolor=blue,anchorcolor=blue,citecolor=blue]{hyperref}
\usepackage{algorithm}
\usepackage{enumerate}
\usepackage{algpseudocode}
\usepackage{verbatim}
\usepackage{graphicx}

\newcommand\C{{\mathbb C}}
\newcommand\Q{{\mathbb Q}}
\newcommand\Z{{\mathbb Z}}

\newcommand\OO{{\mathcal O}}

\newcommand\ord{\mathrm{ord}}
\newcommand\h{\mathrm{h}}

\newcommand\al{\alpha}
\newcommand\be{\beta}
\newcommand\ga{\gamma}
\newcommand\p{\mathfrak{p}}
\newcommand\Arg{\mathrm{Arg}}

\newtheorem{theorem}{Theorem}[section]
\newtheorem{lemma}[theorem]{Lemma}

\theoremstyle{remark}

\numberwithin{equation}{section}

\def\house#1{{%
    \setbox0=\hbox{$#1$}
    \vrule height \dimexpr\ht0+1.4pt width .5pt depth \dp0\relax
    \vrule height \dimexpr\ht0+1.4pt width \dimexpr\wd0+2pt depth \dimexpr-\ht0-1pt\relax
    \llap{$#1$\kern1pt}
    \vrule height \dimexpr\ht0+1.4pt width .5pt depth \dp0\relax}}

\begin{document}

\title[The Skolem Problem for linear recurrence sequences]{Effective results on the Skolem Problem for linear recurrence sequences}

\author{Min Sha}
\address{Department of Computing, Macquarie University, Sydney, NSW 2109, Australia}
\email{shamin2010@gmail.com}



\begin{abstract}
In this paper, given a simple linear recurrence sequence of algebraic numbers, which has either a dominant characteristic root 
or exactly two characteristic roots of maximal modulus, 
we give some explicit lower bounds for the index beyond which every term of the sequence is non-zero. 
 It turns out that this case covers almost all such sequences whose coefficients are rational numbers. 
\end{abstract}

\maketitle

\section{Introduction}

\subsection{Background and motivation}

Linear recurrence sequences (LRS) appear almost everywhere in mathematics and computer science, and they have been studied for a very long time; see \cite{EPSW} for a deep and extensive introduction.
In this paper, we focus on the \textit{Skolem Problem}, which asks whether there is a zero term in a given LRS.

As usual, let $\bar{\Q}$ be the field of all algebraic numbers, which is an algebraic closure of the rational numbers $\Q$. 
Recall that an LRS of \textit{order} $m\ge 1$ is a sequence $\{u_n\}_{n= 0}^{\infty}$ with elements in $\bar{\Q}$ satisfying a recurrence relation
\begin{equation}
\label{sequence}
u_{n+m}=a_{m-1}u_{n+m-1}+\cdots+a_0u_n \quad (n=0,1,2,\ldots),
\end{equation}
where $a_0,\dots,a_{m-1}\in \bar{\Q}$, $a_0 \ne 0$ and $u_j \ne 0$ for at least one
 $j$ in the range $0 \le j \le m-1$. 
 Here, we call $a_0,\ldots,a_{m-1}$ the \textit{coefficients} of the sequence $\{u_n\}$, and the \textit{initial terms} of $\{u_n\}$ are $u_0,\ldots,u_{m-1}$. 

Several crucial properties of the sequence $\{u_n\}$ rely on its 
 \textit{characteristic polynomial}, which is defined as 
\begin{equation*}
\label{polynomial}
f(X)=X^m-a_{m-1}X^{m-1}-\cdots-a_0=\prod_{i=1}^{k}(X-\alpha_i)^{d_i}  \in \bar{\Q}[X]
\end{equation*}
with distinct $\alpha_1,\alpha_2,\ldots,\alpha_k$ (which are called the  \textit{characteristic roots} of the sequence $\{u_n\}$) and $d_i>0$ for $1\le i \le k$. Then, $u_n$ can be expressed as
\begin{equation}
\label{expression}
u_n=\sum_{i=1}^{k}f_i(n)\alpha_i^n,
\end{equation}
where $f_i$ is some polynomial of degree at most  $d_i-1$ ($i=1,2,\ldots,k$). We call the sequence $\{u_n\}$ \textit{simple} if $k=m$ (that is $d_1=\cdots=d_m=1$) and \textit{non-degenerate} if $\alpha_i/\alpha_j$ is not a root of unity for any $i\ne j$ with $1\le i,j\le k$. 
It is well-known that if $\{u_n\}$ is non-degenerate, then there are only finitely many integers $n$ such that $u_n=0$. 
In fact, it has been shown in \cite{DS1} that almost all integer polynomials are non-degenerate. 

The celebrated Skolem-Mahler-Lech Theorem asserts that the \textit{zero set} $\{n: \,  u_n=0\}$ is the union of a finite set and finitely many arithmetic progressions 
(for instance, see \cite[Theorem 2.1]{EPSW}). 
However, all of its existing proofs are in a non-constructive manner.  
Berstel and Mignotte \cite{BM} showed how to obtain all the arithmetic progressions effectively mentioned in the theorem. 
So, it remains to decide the finite part of the zero set, where one must decide whether the finite part is empty or not. 
The \textit{Skolem Problem}, posed in 1930s,  asks whether it is algorithmically decidable that there exists some $n$ such that $u_n=0$. 

There are only few results towards the decidability of the Skolem Problem. For such sequences of order
1 and 2, this problem is relatively straightforward.
Decidability for LRS over $\bar{\Q}$  of orders 3 and 4 is
independently settled positively by Mignotte, Shorey and Tijdeman \cite{MST}, as well as Vereshchagin \cite{Ver}. 
More recently,  the decidability of  the Skolem Problem for integer LRS of order 5 was claimed in \cite{HHHK}, 
and the decidability for rational LRS of any order was claimed in \cite{Litow}, but as pointed out in \cite{OW12}, both are incorrect. 
The Skolem Problem is
also listed as an open problem and discussed by Tao \cite[Section 1.9]{Tao}; see also \cite{OW12} for a survey. 
To taste the difficulty of the problem, we want to point out that Blondel and Portier \cite[Corollary 2.1]{BP} showed that it is NP-hard to decide whether a given integer LRS has a zero. 

Most recently, when the order of $\{u_n\}$ is 2, 3, or 4,  Chonev, Ouaknine and  Worrell \cite[Theorem 2.1]{COW} gave an effective (not explicit) lower bound $N$, 
which roughly is a polynomial function of its coefficients and initial terms, such that $u_n \ne 0$ for any $n > N$; see also \cite[Theorem 19]{Chonev} for a more clear version.   

In this paper, we want to obtain an explicit version for such an upper bound $N$ when the sequence $\{u_n\}$ is simple  
and it has either a dominant characteristic root or exactly two characteristic roots of maximal modulus  (but we don't restrict its order). 
This can be viewed as an explicit version of partial results in \cite[Corollary 1]{MST}.  
 It turns out that this case covers almost all LRS of algebraic numbers whose coefficients are rational numbers.

 \subsection{Main results}
 
 We now present the main results and discuss briefly their proofs and coverage. 
 
For any polynomial  $f(X) \in \bar{\Q}[X]$ of degree $m$, 
let $\delta_f$ be the smallest positive integer such that all the coefficients of the polynomial $\delta_f f(X)$ are algebraic integers. 
Denote $\delta_f f(X)$ by $f^*(X)$ and write 
\begin{equation}    \label{eq:f*}
f^*(X)=\sum_{i=0}^{m}a_i^*X^{i}. 
\end{equation}
For any non-zero $\alpha \in \bar{\Q}$, let $\h(\alpha)$ be the (Weil) absolute logarithmic height of $\alpha$. 
Let $e$ be the base of the natural logarithm. 

\begin{theorem}
\label{thm:dom}
Let $\{u_n\}$ be a simple LRS of algebraic numbers defined by \eqref{sequence} of order $m\ge 2$, and let $f(X)$ be its characteristic polynomial. 
Suppose that $f(X)$ has a dominant root. 
Let $d$ be the degree of the Galois closure of the field $\Q(a_0,a_1,\ldots, a_{m-1})$ over $\Q$, 
 and let $D$ be the degree of the number field generated by $u_0,\ldots,u_{m-1}$ over $\Q$.  
Let  $f^*(X)=\sum_{i=0}^{m}a_i^*X^{i}$ be defined as in \eqref{eq:f*}, and let 
$$
I(f^*) = 2^{dm}(m+1)^{d/2} \prod_{i=0}^{m} \exp(d\h(a_i^*)) 
$$
and 
$$
J(f^*)=2^{dm(dm-1)/4}(dm+1)^{-d^3m^3/4+3dm/4-7}I(f^*)^{-d^3m^3/2+d^2m^2+dm/2-11}. 
$$
Denote 
$$
B(u)=m! \cdot d D\Big( \sum_{i=0}^{m-1}\h(u_{i})+3dm^2\sum_{i=0}^{m}\h(a_i^*)+3dm^2\log(m+1)\Big).
$$
Then, if $n > N_1(u)$, we have $u_n \ne 0$, where 
$$
N_1(u) = (2B(u) + \log m)(1+H(f))J(f^*)^{-1}. 
$$
If furthermore $f$ is a real polynomial, then in the lower bound $N_1(u)$, $J(f^*)$ can be replaced by 
$$
 2^{-dm(dm-1)(dm-2)/2} (dm+1)^{-dm(dm-1)-1/2}I(f^*)^{-2dm(dm-1)-1}.  
$$
\end{theorem}

For the lower bound $N_1(u)$ in Theorem \ref{thm:dom}, if we fix $f$ (that is, fixing the coefficients $a_0,a_1,\ldots, a_{m-1}$), then we have 
$$
N_1(u) \ll_f D \Big(\sum_{i=0}^{m-1}\h(u_{i}) + 1 \Big).  
$$
Here, we use the Vinogradov symbol $\ll$. Recall that the
assertion $U \ll V$  is equivalent to the inequality $|U|\le cV$ with some absolute constant $c>0$.
  To emphasise the dependence of the implied
constant $c$ on some parameter $\rho$, we write  $U \ll_{\rho} V$.

\begin{theorem}
\label{thm:nondom}
Let $\{u_n\},f,f^*,d, D,I(f^*), J(f^*), B(u)$ be defined as in Theorem \ref{thm:dom}. 
Suppose that $f$ has exactly two roots of maximal modulus, 
and moreover their quotient is not a root of unity. 
Denote 
$$
C(u) = 2^{40} (m!\cdot dD)^2 \pi (2B(u)+\pi) \log (I(f^*)) \log(m! \cdot edD). 
$$
Then, if $n > N_2(u)$, we have $u_n \ne 0$, where 
$$
N_2(u) = 4C(u)I(f^*)J(f^*)^{-1} \log \big( 2C(u)I(f^*)J(f^*)^{-1}\big). 
$$
\end{theorem}

For the lower bound $N_2(u)$ in Theorem \ref{thm:dom}, fixing $f$, we have 
$$
N_2(u) \ll_f D^3 (\log(D+1))^2 \Big( \sum_{i=0}^{m-1}\h(u_{i}) + 1\Big) \log \Big(\sum_{i=0}^{m-1}\h(u_{i}) + 2 \Big) .  
$$

We will prove Theorems~\ref{thm:dom} and \ref{thm:nondom} in Sections~\ref{sec:dominant} and \ref{sec:two} respectively 
after making some preparations in Section~\ref{sec:basic}.  
The approach of the proofs is straightforward. 
For any simple LRS $\{u_n\}$ of order $m$, as in \eqref{expression} we can write 
$u_n = \sum_{j=1}^{m} b_j \alpha_j^n$, 
then we try to find a lower bound for the index beyond which the absolute value of the part of the summation 
related to the roots of maximal modulus is greater than the absolute value of the rest of the summation. 
For this, we need to obtain lower bounds on separating the absolute values $|\alpha_1|, \ldots, |\alpha_m|$, 
and estimate the sizes of the coefficients $b_1, \ldots, b_m$. 
Especially, when there are two characteristic roots of maximal modulus, 
we need to employ Matveev's bound on linear forms in the logarithms of algebraic numbers. 

Finally, we say something about the coverage of the main results. 

By \cite[Theorem 1.1]{DS2}, almost all monic integer polynomials in $\Z[X]$ have a dominant root. 
In other words, Theorem~\ref{thm:dom} covers almost all the linear recurrence sequences of algebraic numbers 
whose coefficients are rational integers. 

Besides, by \cite[Theorem 4]{DS1} and \cite[Theorem 1.1]{DS3}, almost all integer polynomials in $\Z[X]$ (not necessarily  monic) 
are non-degenerate and have either a dominant root or exactly two roots of maximal modulus. 
Note that each monic polynomials in $\Q[X]$ can become an integer polynomial by multiplying some positive integer. 
So, we can say that Theorem~\ref{thm:dom} and Theorem~\ref{thm:nondom} cover almost all the linear recurrence sequences of algebraic numbers 
whose coefficients are rational numbers.

\section{Preliminaries}
\label{sec:basic}

%In this section, we gather some concepts and results which are used later on.

\subsection{Height and Mahler measure}

Given a polynomial
$$
f(X)=a_mX^m+a_{m-1}X^{m-1}+\cdots+a_0=a_m (X-\al_1)\cdots (X-\al_m) \in \C[X]
$$
of degree $m\ge 1$, we assume that the roots $\al_1,\dots,\al_m$
(listed with multiplicities)
are labelled so that $|\al_1 | \ge |\al_2| \ge \cdots \ge |\al_m |$. In case $|\al_1|=\cdots =|\al_r|>|\al_{r+1} |$, 
we say that \textit{$f$ has exactly $r$ roots of maximal modulus}. 
If $r=1$, we say that $f$ has a \textit{dominant root} (that is, $\alpha_1$).
 Clearly, the dominant root is a real number if $f$ is a real polynomial. 

For the polynomial $f$,  its \textit{length} is defined by 
$$
L(f)=|a_0|+\cdots+ |a_m|,
$$
 its {\it height} by 
 $$
 H(f)=\max_{0 \leq i \leq m} |a_i|,
 $$
and its {\it Mahler measure} by
$$
M(f)=|a_m| \prod_{i=1}^{m} \max\{1,|\al_i|\}.
$$
These quantities are related by the following inequality
\begin{equation}
\label{Mahler}
H(f) 2^{-m} \leq M(f) \leq H(f) \sqrt{m+1}, 
\end{equation}
for instance, see \cite[(3.12)]{Waldschmidt}. 
 If furthermore $f\in \Z[X]$ is square-free, then for any two distinct roots $\al, \be$ of $f$, Mahler's inequality \cite{Mahler} asserts that  
 \begin{equation}   
 \label{Mahler2}
 |\al - \be|  > \sqrt{3} m^{-m/2-1}M(f)^{-m+1}. 
 \end{equation}
 Given another polynomial $g \in \C[X]$, by definition we have 
$$
M(fg)=M(f)M(g). 
$$

 Accordingly, for a non-zero algebraic number $\al$, its \emph{Mahler measure} $M(\al)$ is defined as the Mahler measure of its minimal polynomial $f$ over the integers $\Z$, that is, $M(\al)=M(f)$.

For a number field $K$ of degree $d$ (over $\Q$), we denote by $M_{K}$ the set of all valuations $v$ of $K$ extending the standard infinite and
$p$-adic valuations of the rational numbers $\Q$: $|2|_{v}=2$ if $v\in M_{K}$ is Archimedean, and $|p|_{v}=p^{-1}$ if $v$ extends the $p$-adic valuation of $\Q$. In particular, if the valuation $v$ of $K$ corresponds to a prime ideal $\p$ of $K$ lying above a prime number $p$, we also denote the valuation $|\,\,|_v$ by $|\,\,|_\p$, then for any $\al \in K$ we have
$$
|\al|_\p = p^{-\ord_\p(\al)/e_{\p}},
$$
where $\ord_\p(\al)$ is the exponent of $\p$ appearing in the prime decomposition of the fractional ideal $\al \OO_K$, $\OO_K$ is the ring of integers of $K$,  and $e_\p$ is the ramification index of $\p$ over $p$.
For any $v\in M_{K}$, let $K_{v}$ be the completion of $K$ with respect to the valuation $v$, and
let $d_{v}=[K_{v}:\Q_{v}]$ be the local degree of $v$. When the valuation $v$ corresponds to a prime ideal $\p$ lying above a prime number $p$, we also denote $K_v$ by $K_\p$ and $\Q_v$ by $\Q_p$, respectively. 

For the above number field $K$, the \emph{(Weil) absolute logarithmic height} of any non-zero $\alpha\in K$ is
defined by
\begin{equation}
\label{Weil}
\h(\alpha)=d^{-1}\sum_{v\in M_{K}}d_{v}\log\max\{1, |\alpha|_{v}\}.
\end{equation}
 Moreover, if $\alpha$ is of degree $d$, we have
\begin{equation}
\label{Mahler-Weil}
 \h(\al)= d^{-1} \log M(\al); 
\end{equation}
see \cite[Lemma 3.10]{Waldschmidt}. 

Given non-zero $\al \in K$, in view of \eqref{Weil} and $\h(\al)=\h(\al^{-1})$, for any valuation $v\in M_K$ we have 
\begin{equation}
\label{Liouville}
| \log |\al |_v | \le d\h(\al)/d_v \le d\h(\al).
\end{equation}

In the sequel, we  use the following formulas without special reference (see, e.g., \cite{Waldschmidt}). For any $n\in \mathbb{Z}$ and $\be_{1},\cdots,\be_{k},\gamma\in \bar{\mathbb{Q}}$, we have
\begin{align*}
&\h(\be_{1}+\cdots+\be_{k})\le \h(\be_{1})+\cdots+\h(\be_{k})+\log k,\\
&\h(\be_{1}\cdots \be_{k})\le \h(\be_{1})+\cdots+\h(\be_{k}),\\
&\h(\gamma^{n})=|n|\h(\gamma),\\
&\h(|\ga|) \le \h(\ga), \\
& \h(\zeta)=0 \textrm{\quad for any root of unity $\zeta\in \bar{\Q}$}.
\end{align*}

We also need the following result, which is exactly~\cite[Lemma 3.7]{Waldschmidt}.

\begin{lemma}
\label{antra}
Let $f \in \Z[X_1,\ldots,X_n]$ be a non-zero polynomial in $n$ variables.  Then, for any algebraic numbers $\ga_1,\ldots,\ga_n$, we have
$$
\h(f(\ga_1,\ldots,\ga_n)) \leq \log L(f)+\sum_{i=1}^n \h(\ga_i) \deg_{X_i} f,
$$
where $\deg_{X_i} f$ is the partial degree of $f$ with respect to $X_i$.
\end{lemma}

 \subsection{Absolute root separation}
 
 Mahler has given a celebrated result in \cite{Mahler} on separating distinct roots of a polynomial in $\Z[X]$.   
 For our purpose, we need a result on separating the absolute values of the roots of a polynomial with coefficients as algebraic integers. 

The following lemma is a classical result due to Cauchy; see  \cite[Proposition 2.5.9]{Mignotte}.

\begin{lemma}
\label{Cauchy}
Let $f(X) \in \C[X]$ be a polynomial of degree $m \ge 1$ defined by
$$
f(X)=a_mX^m+a_{m-1}X^{m-1}+\cdots+a_0,  \quad a_m \ne 0. 
$$
Then, for any root $z$ of $f$, we have 
\begin{equation*}
|z|<1+\frac{1}{|a_m|}\max\{|a_0|,\ldots,|a_{m-1}|\}.
\notag
\end{equation*}
\end{lemma}

We reproduce \cite[Lemma 2.4 and Lemma 2.5]{DS2} as follows.

\begin{lemma}
\label{quadratic}
Let $f(X)\in\Z[X]$ be a quadratic polynomial. Suppose that $f$ has two real roots $\al$ and $\be$ with $|\al| \ne |\be|$. Then, we have
$$
||\al|-|\be|| \ge H(f)^{-1}.
$$
\end{lemma}

\begin{lemma}
Let $f(X)\in\Z[X]$ be a polynomial of degree $m \ge 2$, and let $\alpha$ and $\beta$ be two roots of $f$ satisfying $|\alpha| \ne |\beta|$. Then,
\begin{equation}
\label{distance1}
\left||\alpha|-|\beta|\right|>
2^{m(m-1)/4}(m+1)^{-m^3/4+3m/4-3}H(f)^{-m^3/2+m^2+m/2-2}
\end{equation}
if both $\al$ and $\be$ are non-real.
If, furthermore, $\al$ is real and $\be$ is non-real, then
\begin{equation}
\label{distance2}
\left||\alpha|-|\beta|\right| \ge 2^{-m(m-1)(m-2)/2} (m+1)^{-m(m-1)-1/2}H(f)^{-2m(m-1)-1}.
\end{equation}
Finally, if both $\al$ and $\be$ are real, then
\begin{equation}
\label{distance3}
||\alpha|-|\beta||>(2m+1)^{-3m}H(f)^{2-4m}.
\end{equation}
\end{lemma}

We remark that there is an improvement upon \eqref{distance3} in \cite{BDPS} for real roots under some further conditions. 
Note that for large enough $m$, \eqref{distance3} is better than \eqref{distance2}, and \eqref{distance2} is better than \eqref{distance1}. However, for small integer $m$, this might be not true. 
 For simplicity, we put them together into two uniform forms.

\begin{lemma} 
\label{lem:dist1}
Let $f(X)\in\Z[X]$ be a polynomial of degree $m\ge 2$, and let $\alpha$ and $\beta$ be two roots of $f$ satisfying $|\alpha| \ne |\beta|$. Then,
\begin{equation}
\label{udis1}
\left||\alpha|-|\beta|\right|>
2^{m(m-1)/4}(m+1)^{-m^3/4+3m/4-7}H(f)^{-m^3/2+m^2+m/2-11};
\end{equation}
if furthermore $\al$ is real, then
\begin{equation}
\label{udis2}
\left||\alpha|-|\beta|\right| \ge 2^{-m(m-1)(m-2)/2} (m+1)^{-m(m-1)-1/2}H(f)^{-2m(m-1)-1}.
\end{equation}
\end{lemma}

\begin{proof}
By Lemma \ref{quadratic}, we can assume that $m \ge 3$.    
We first prove \eqref{udis2}. 
Notice that the inequality \eqref{udis2} is the same as \eqref{distance2}. 
For any $m \ge 4$, we have 
$$
2^{-m(m-1)(m-2)/2}(m+1)^{-m(m-1)-1/2} < 2^{-3m} (m+1)^{-3m} < (2m+1)^{-3m}
$$
and $-2m(m-1)-1 < 2-4m$, and so \eqref{distance3} is included in \eqref{udis2} when $m \ge 4$. 
We now consider $m=3$ individually. 
Assume that $f$ has two real roots $\al$ and $\be$ such that  $|\al| \ne |\be|$. 
Then, its third root, say $\gamma$, is also real. 
If $\gamma \ne \pm \al$ and $\gamma \ne \pm \be$, then by \cite[Theorem 1]{BDPS} we have 
$$
||\al|-|\be|| \ge 2^{-5.5}H(f)^{-2}, 
$$
which is certainly included in \eqref{udis2} by setting $m=3$. 
Now, if $\gamma = \pm \al$ or $\gamma = \pm \be$, then the polynomial $f(X)f(-X)$ has a multiple root ($\al$ or $\be$). 
Let $g(X)$ be the squarefree part of $f(X)f(-X)$. Then, we have $\deg g \le 5$. 
Note that $\pm \al$ and $\pm \be$ are real roots of $g$. 
So, the value $||\al|-|\be||$ is in fact equal to the absolute value of the difference of two distinct roots of $g$. 
Thus, applying \eqref{Mahler} and \eqref{Mahler2} to $g$, we obtain 
$$
||\al|-|\be|| > \sqrt{3} \cdot 5^{-3.5} M(g)^{-4} 
\ge  \sqrt{3} \cdot 5^{-3.5} M(f)^{-8} \ge \sqrt{3} \cdot 5^{-3.5} \cdot 2^{-8}H(f)^{-8}, 
$$
which is also included in \eqref{udis2} by setting $m=3$. 
This completes the proof of \eqref{udis2}. 

Now, we want to prove \eqref{udis1}. 
By \eqref{udis2}, we only need to prove that both \eqref{distance1} and \eqref{distance2} are included in \eqref{udis1}. 
Note that \eqref{distance1} is automatically contained in \eqref{udis1}. 
It remains to show that \eqref{distance2} is included in \eqref{udis1}. 
First, for $m=3, 4$ or $5$, by direct computation we have 
$$
-m^3/2+m^2+m/2-11 \le -2m(m-1)-1, 
$$
and for $m \ge 6$, we obtain 
$$
-m^3/2+m^2+m/2-11 \le -2m^2 + m/2 - 11 < -2m(m-1)-1, 
$$
and thus, for any $m\ge 3$ we have 
\begin{equation} 
\label{eq:H1}
H(f)^{-m^3/2+m^2+m/2-11} \le H(f)^{-2m(m-1)-1}.  
\end{equation}
On the other hand, for $m=3, 4, 5$ or $6$, by direct computation we have 
$$
2^{m(m-1)/4}(m+1)^{-m^3/4+3m/4-7} < 2^{-m(m-1)(m-2)/2} (m+1)^{-m(m-1)-1/2}, 
$$
and for any $m\ge 7$, it is easy to see that 
\begin{equation}
\label{eq:2m}
2^{m(m-1)/4}(m+1)^{-m^3/4+3m/4-7} \le 2^{-m(m-1)(m-2)/2} (m+1)^{-m(m-1)-1/2}. 
\end{equation}
Indeed, to obtain \eqref{eq:2m} it is equivalent to show 
$$
2^{m(m-1)(m-2)/2+m(m-1)/4}\le  (m+1)^{m^3/4-3m/4-m(m-1)+13/2}, 
$$
which follows from (note that $m \ge 7$)
\begin{align*}
2^{m(m-1)(m-2)/2+m(m-1)/4} & \le (m+1)^{m(m-1)(m-2)/6+m(m-1)/12} \\ 
& < (m+1)^{m^3/4-3m/4-m(m-1)+13/2}. 
\end{align*}
So, for any $m \ge 3$, we obtain 
\begin{equation}
\label{eq:H2}
2^{m(m-1)/4}(m+1)^{-m^3/4+3m/4-7} \le 2^{-m(m-1)(m-2)/2} (m+1)^{-m(m-1)-1/2}. 
\end{equation}
Hence, combining \eqref{eq:H1} with \eqref{eq:H2}, we deduce that \eqref{distance2} is included in \eqref{udis1}. 
This completes the proof of \eqref{udis1}. 
\end{proof}

Moreover, we can extend the above lemma to polynomials whose coefficients are algebraic integers. 
For this, we need a simple preparation. 

\begin{lemma}
\label{lem:Hg}
Let $f(X) = a_mX^m + \cdots + a_1X + a_0$ be a polynomial of degree $m\ge 2$, where all the coefficients are algebraic integers. 
Let $K$ be a finite Galois extension over $\Q$ containing the field $\Q(a_0,a_1,\ldots, a_m)$. 
Let $d=[K:\Q]$, and let $G$ be the Galois group of $K$ over $\Q$. 
Then, we have 
$$
M(\prod_{\sigma \in G} \sigma(f)) \le (m+1)^{d/2} \prod_{i=0}^{m} \exp(d\h(a_i)), 
$$
and 
$$
H(\prod_{\sigma \in G} \sigma(f)) \le 2^{dm}(m+1)^{d/2} \prod_{i=0}^{m} \exp(d\h(a_i)). 
$$
\end{lemma}

\begin{proof} 
For each $0 \le i \le m$, let $d_i$ be the degree of $a_i$ over $\Q$. 
Using \eqref{Mahler}, we have 
\begin{equation*}
\label{eq:Hg}
\begin{split}
M(\prod_{\sigma \in G} \sigma(f))   = \prod_{\sigma \in G}  M(\sigma(f)) 
& \le  \prod_{\sigma \in G}  \sqrt{m+1}H(\sigma(f)) \\
& = (m+1)^{d/2} \prod_{\sigma \in G}  \max_{0\le i \le m} |\sigma(a_i)| \\
& \le (m+1)^{d/2} \prod_{i=0}^{m} M(a_i)^{d/d_i} \\
& = (m+1)^{d/2} \prod_{i=0}^{m} \exp(d\h(a_i)),
\end{split}
\end{equation*} 
where we also use the assumption that the coefficients $a_0,a_1,\ldots, a_m$ are algebraic integers. 
This completes the proof of the first inequality. 
The second inequality follows from the first one and \eqref{Mahler}.  
\end{proof}

Now, we are ready to extend Lemma~\ref{lem:dist1}. 

\begin{lemma}
\label{lem:dist2}
Let $f(X) = a_mX^m + \cdots + a_1X + a_0$ be a polynomial of degree $m\ge 2$, where all the coefficients are algebraic integers. 
Let $K$ be the Galois closure of the field $\Q(a_0,a_1,\ldots, a_m)$ over $\Q$. 
Put $d=[K: \Q]$, and let $G$ be the Galois group of $K$ over $\Q$. 
Denote 
$$
I(f) = 2^{dm}(m+1)^{d/2} \prod_{i=0}^{m} \exp(d\h(a_i)). 
$$
If $\alpha$ and $\beta$ are two roots of the polynomial $\prod_{\sigma \in G} \sigma(f)$ 
satisfying $|\alpha| \ne |\beta|$, then 
\begin{equation}
\label{eq:udis1}
\left||\alpha|-|\beta|\right|>
2^{dm(dm-1)/4}(dm+1)^{-d^3m^3/4+3dm/4-7}I(f)^{-d^3m^3/2+d^2m^2+dm/2-11};
\end{equation}
if furthermore $\al$ is real, then
\begin{equation}
\label{eq:udis2}
\left||\alpha|-|\beta|\right| \ge 2^{-dm(dm-1)(dm-2)/2} (dm+1)^{-dm(dm-1)-1/2}I(f)^{-2dm(dm-1)-1}.
\end{equation}
\end{lemma}

\begin{proof}
By assumption, the polynomial $g=\prod_{\sigma \in G} \sigma(f)$ is a polynomial in $\Z[X]$. 
Clearly, $\deg g = dm$, because $|G|=d$. 
By Lemma~\ref{lem:Hg}, we have $H(g) \le I(f)$. 
Then, applying Lemma~\ref{lem:dist1} to $g$ we obtain 
the desired results. 
\end{proof}

We remark that in Lemma \ref{lem:dist2}, if the degree of each coefficient $a_i$ over $\Q$ is $d_i$, $i=0,1,\ldots, m$, then we have 
$d \le \prod_{i=0}^{m} d_i !$.

\subsection{Bounding coefficients}

For further deductions, we need to estimate the coefficients in \eqref{expression}
when the sequence $\{u_n\}$ is a simple LRS of algebraic numbers.

\begin{lemma}
\label{lem:hbj}
Let $\{u_n\}$ be a simple LRS of algebraic numbers of order $m \ge 2$  defined by \eqref{sequence}. 
Let $f(X)$ be its characteristic polynomial, 
and define the polynomial $f^*(X)$ as in \eqref{eq:f*}. 
Write $u_n$ as
$$
u_n=\sum_{j=1}^{m}b_j\al_j^n,
$$
where $\al_1,\ldots,\al_m$ are distinct roots of $f$ and all $b_j$ are non-zero. Then, for any $1\le j \le m$ we have
$$
\h(b_j) \le  \sum_{i=0}^{m-1}\h(u_{i})+2m\sum_{k \ne j}\h(\al_k)+m^2\h(\al_j)+m(2m-3)\log 2 + \log m. 
$$
Let $d$ be the degree of the Galois closure of the field $\Q(a_0,a_1,\ldots, a_{m-1})$ over $\Q$.  
Then, we have 
$$
\h(b_j) < \sum_{i=0}^{m-1}\h(u_{i})+3dm^2\sum_{i=0}^{m}\h(a_i^*)+3dm^2\log(m+1). 
$$
\end{lemma}

\begin{proof}
Here, we follow the arguments in the proof of \cite[Theorem 3.1]{DSS}.

Notice that
\begin{equation}
\label{eq:matrix}
(u_0,u_1, \ldots, u_{m-1})
=
(b_1,b_2, \ldots, b_m)
\left( \begin{array}{cccc}
1 & \al_1 & \ldots & \al_1^{m-1} \\
1 & \al_2 & \ldots & \al_2^{m-1} \\
\vdots & \vdots & \ldots & \vdots\\
1 & \al_m & \ldots & \al_m^{m-1} \\
\end{array} \right),
\end{equation}
and $\al_1, \ldots, \al_m$ are distinct.
To solve the above system of $m$ linear equations in $m$ unknowns $b_1,\ldots,b_m$, we denote the appearing Vandermonde matrix by
$V=\left(\al_i^{j-1}\right)_{1 \leq i,j \leq m}$. By~\cite[Formula (6)]{Klinger1967}, the inverse of $V$ is given by $V^{-1}=\big(w_{ij}\big)_{1\le i,j \le m}$, where
$$
w_{ij}=
 \frac{(-1)^{i+j} \sigma_{m-i}(\al_1,\ldots,\widehat{\al_{j}},\ldots,\al_m)}{\prod\limits_{l=1}^{j-1}(\al_{j}-\al_l)\prod\limits_{k=j+1}^{m}(\al_{k}-\al_{j})}
$$
 and $\sigma_{k}(\al_1,\ldots,\widehat{\al_{j}},\ldots,\al_m)$ stands for the $k$-th elementary symmetric function in the $m-1$ variables
$\al_1,\ldots,\al_m$ without $\al_{j}$; for instance, in the case $j=m$, we have
$\sigma_{1}(\al_1,\ldots,\al_{m-1})=\al_1+\cdots+\al_{m-1}$ and $\sigma_{m-1}(\al_1,\ldots,\al_{m-1})=\al_1\cdots \al_{m-1}$.

So, for any $j$ with $1\le j \le m$ we have
\begin{equation}\label{uyt}
 b_{j} = \sum_{i=1}^m u_{i-1} w_{ij}.
 \notag
\end{equation}
Since $\sigma_{m-i}(\al_1,\ldots,\widehat{\al_{j}},\ldots,\al_m)$
is a polynomial with coefficients $1$ in $m-1$ variables
 $\al_1,\ldots,\al_m$ (without $\al_{j}$) of degree $m-i$, length ${m-1 \choose m-i}$, and degree $1$ in each variable $\al_k$, $k \ne j$, by Lemma~\ref{antra} we find that
$$
\h(\sigma_{m-i}(\al_1,\ldots,\widehat{\al_{j}},\ldots,\al_m))
\leq \log {m-1 \choose m-i}+\sum_{k \ne j} \h(\al_k).
$$
On the other hand, we observe that
\begin{align*}
\h\big(\prod_{k \ne j} (\al_k-\al_j)\big) & \le \sum_{k \ne j} \h(\al_k-\al_j) \\
& \le  \sum_{k \ne j} \big( \h(\al_k)+\h(\al_j)+\log 2 \big) \\
& =   \sum_{k \ne j} \h(\al_k)+(m-1)\h(\al_j)+(m-1)\log 2.
\end{align*}
Thus, we obtain
$$
\h(w_{ij})\le 2\sum_{k \ne j} \h(\al_k)+(m-1)\h(\al_j)+ (m-1)\log 2 + \log {m-1 \choose m-i} .
$$

Hence,  for $1\le j \le m$ we conclude that
\begin{equation} 
\label{eq:hbj}
\begin{split}
\h(b_j) & \le \sum_{i=1}^{m} (\h(u_{i-1})+\h(w_{ij}))+\log m\\
&\leq \sum_{i=0}^{m-1}\h(u_{i})+2m\sum_{k \ne j}\h(\al_k)+m(m-1)\h(\al_j) \\
& \qquad +m(2m-3)\log 2 + \log m,
%&< \sum_{i=0}^{m-1}\h(u_{i})+2m\sum_{k \ne j}\h(\al_k)+m^2\h(\al_j)+\frac{3}{2}m^2-\frac{1}{2}m-1,
\end{split}
\end{equation} 
where we also use the fact that the binomial coefficient ${m-1 \choose m-i} \le 2^{m-2}$ for any $1\le i \le m$. 
This gives the first desired upper bound. 

Now, we need to estimate $\h(\al_i)$ for each $1\le i \le m$.  
By definition and using \eqref{Mahler-Weil} and Lemma~\ref{lem:Hg}, we obtain
\begin{equation} 
\label{eq:hai} 
\begin{split}
\h(\al_i) \le \log M(\al_i) & \le \log M(\prod_{\sigma \in G} \sigma(f^*)) \\
& \le d \sum_{i=0}^{m} \h(a_i^*) + \frac{d}{2}\log(m+1). 
\end{split}
\end{equation} 

Finally, combining \eqref{eq:hbj} with \eqref{eq:hai} we have 
\begin{equation*}
\h(b_j) < \sum_{i=0}^{m-1}\h(u_{i})+3dm^2\sum_{i=0}^{m}\h(a_i^*)+3dm^2\log(m+1). 
\end{equation*}
This completes the proof. 
\end{proof}

\subsection{Linear form in the logarithms of algebraic numbers}

One key technical tool in this paper is Baker's inequality on linear form in the logarithms of algebraic numbers. Here we restate one of its explicit forms due to Matveev \cite[Corollary 2.3]{Matveev}.

First, recall that for a non-zero complex number $z$, the principal value of the natural logarithm of $z$ is
$$
\log z = \log |z| + \sqrt{-1} \cdot \Arg(z),
$$
where $\Arg(z)$ is the principal value of the argument of $z$ ($-\pi < \Arg(z) \le \pi$). Note that the definition here coincides with the natural logarithm of positive real numbers.   We also want to indicate that the identity $\log (z_1z_2)=\log z_1 + \log z_2$ can fail in our setting. 

Let
$$
\Lambda = b_1 \log \al_1 + b_2 \log \al_2 + \cdots + b_k \log \al_k,
$$
where $k\ge 2$, $b_1,\ldots,b_k \in \Z$, and $\al_1,\ldots, \al_k$ are non-zero elements of a number field $K$. Let $D=[K:\Q]$ and $B=\max \{ |b_1|,\ldots, |b_k| \}$. For any $1 \le j \le k$,  choose a real number $A_j$ such that
$$
A_j \ge \max \{ D\h(\al_j), |\log \al_j |, 0.16 \}.
$$

Suppose that $\Lambda \ne 0$. Then, we have
\begin{equation}
\label{Baker}
\log |\Lambda| > -2^{6k+20}D^2A_1\cdots A_k \log (eD) \log (eB),
\end{equation}
where $e$ is the base of the natural logarithm.

We remark that we in fact only need a lower bound on linear forms in three logarithms. 
However, all the existing lower bounds on linear forms in three logarithms are under some extra conditions, 
which do not always hold in our case (see, for instance, the best known estimate \cite[Theorem 2]{Mig}).

\section{Proof of Theorem~\ref{thm:dom}}
\label{sec:dominant}

Let $\al_1,\al_2,\ldots,\al_m$ be the roots of $f$  such that $|\al_1|>|\al_j|$ for any $2\le j \le m$. Note that they are all distinct and also the roots of $f^*$. 

Then, by \eqref{eq:udis1} and the definition of $J(f^*)$, for any $2\le j \le m$ we have
\begin{equation}
\label{eq:diff1}
|\al_1|-|\al_j| > J(f^*).
\end{equation}

As mentioned before, for any integer $n\ge 0$, $u_n$ can be expressed as
$$
u_n=\sum_{j=1}^{m}b_j\alpha_j^n,
$$
where $b_1,\ldots,b_m$ are all non-zero complex numbers.
Now, we want to find a lower bound beyond which the index $n$ satisfies 
\begin{equation}
\label{eq:b1a}
|b_1\al_1^n|> \sum_{j=2}^{m}|b_j\alpha_j^n|.
\end{equation}
Then, $u_n \ne 0$ when the index $n$ is greater than this lower bound. This will complete the proof. 
Note that  it is equivalent to require that
$$
|b_1|> \sum_{j=2}^{m}|b_j|\left(|\alpha_j|/|\al_1|\right)^n,
$$
which, by \eqref{eq:diff1}, is implied in the inequality
\begin{equation}
\label{eq:b1}
|b_1| > \left(1-J(f^*)/|\al_1|\right)^n\sum_{j=2}^{m}|b_j|.
\end{equation}

On the other hand, for any $1\le j \le m$, by \eqref{Liouville} we know that
$$
| \log |b_j | |\le [\Q(b_j):\Q]\h(b_j).
$$
Since $b_j \in \Q(u_0,\ldots,u_{m-1},\al_1,\ldots,\al_m)$ by \eqref{eq:matrix}, we have $[\Q(b_j):\Q]\le m! \cdot dD$. So
$$
| \log |b_j | |\le m!\cdot dD \h(b_j).
$$
Using Lemma~\ref{lem:hbj} and by the definition of $B(u)$, we get
$$
| \log |b_j | |\le B(u),
$$
that is
\begin{equation}
\label{eq:|bj|}
\exp(-B(u))\le |b_j| \le \exp(B(u))
\end{equation}
for any $1\le j \le m$.

Thus, by \eqref{eq:|bj|}, the inequality \eqref{eq:b1} is implied in the following inequality
$$
\exp(-B(u)) > m \exp(B(u)) \left(1-J(f^*)/|\al_1|\right)^n,
$$
which is equivalent to
$$
n > \frac{2B(u)+ \log m}{ -\log (1-J(f^*)/|\al_1|)}.
$$
By Lemma \ref{Cauchy}, we have $|\al_1| < 1+ H(f)$. So,  it suffices to ensure that
$$
n > \frac{2B(u)+ \log m}{ -\log (1-J(f^*)/(1+H(f)))}.
$$
Using the Taylor expansion $-\log (1-x)=x+x^2/2+x^3/3+ \cdots$ for $|x|<1$, it suffices to require that
\begin{equation*}
n > \frac{2B(u)+ \log m}{ J(f^*)/(1+H(f))}.
\end{equation*} 
Thus, we get the desired lower bound $N_1(u)$ implying the inequality \eqref{eq:b1a}. 
This completes the proof of the first part. 

Finally, if $f$ is a real polynomial, then its dominant root $\alpha_1$ is a real root, and so in the inequality \eqref{eq:diff1} we use \eqref{eq:udis2} 
instead of \eqref{eq:udis1}. 
This in fact gives the second result and completes the proof.

\section{Proof of Theorem~\ref{thm:nondom}}
\label{sec:two}

Under the assumptions, we must have $m \ge 3$. 
Let $\al_1,\al_2,\ldots,\al_m$ be the roots of $f$ such that $|\al_1|=|\al_2|>|\al_j|$ for any $3\le j \le m$. 
Note that they are also the roots of $f^*$. 
By \eqref{eq:udis1} and the definition of $J(f^*)$, for any $3\le j \le m$ we have
$$
|\al_1|-|\al_j| > J(f^*). 
$$
Note that for any integer $n\ge 0$, $u_n$ can be expressed as
$$
u_n=\sum_{j=1}^{m}b_j\alpha_j^n,
$$
where $b_1,\ldots,b_m$ are all non-zero complex numbers. 

In the sequel, we want to find a lower bound beyond which the index $n$ satisfies  
\begin{equation}
\label{eq:b12a}
|b_1\al_1^n+b_2\al_2^n|> \sum_{j=3}^{m}|b_j\alpha_j^n|.
\end{equation}
So, whenever the index $n$ is greater than this lower bound, we have $u_n \ne 0$. This will complete the proof. 

The key step is to get a lower bound for the left-hand side of \eqref{eq:b12a} by using Baker's inequality on linear form \eqref{Baker}. 
Then, let the right-hand side of \eqref{eq:b12a} be less than the lower bound, this can give the desired lower bound for the index $n$. 

For any $n\ge 0$, we have
\begin{equation}
\label{eq:initial}
|b_1\al_1^n+b_2\al_2^n|
=|b_1\al_1^n| \cdot \left|(-1)\cdot \frac{b_2}{b_1}\cdot (\frac{\al_2}{\al_1})^n-1 \right| 
\end{equation}
Here, for $n\ge 0$ we put
$$
\Delta_n = (-1)\cdot \frac{b_2}{b_1}\cdot (\frac{\al_2}{\al_1})^n-1,
$$
and 
\begin{equation}
\label{eq:LaDe}
\Lambda_n = \log (\Delta_n +1). 
\end{equation}
Then, by definition, there exists an  integer $a$ such that 
$$
|a| \le n+2
$$ 
and 
$$
\Lambda_n = a \log (-1) + \log (b_2/b_1) + n \log (\al_2/\al_1), 
$$
which gives a linear form in the logarithms of algebraic numbers. 

In the following, we assume that 
\begin{equation}
\label{eq:1/2}
|\Delta_n| \le 1/2.
\end{equation}
If this is not true, then later on one can see that this implies much better results; see \eqref{eq:b1b2}.

Notice that for any complex number $z$ with $0<|z| \le r <1$, using the Taylor expansion, we have 
\begin{align*}
|\log (1+z)| & = |z-\frac{z^2}{2}+\frac{z^3}{3}- \cdots | \\
& \le (1+\frac{r}{2}+\frac{r^2}{3}+\cdots ) |z| =\frac{|\log (1-r)|}{r}|z|.
\end{align*}
 Using this estimate together with \eqref{eq:LaDe} and \eqref{eq:1/2}, we obtain
\begin{equation}
\label{eq:tran}
\frac{1}{2} |\Lambda_n | = \frac{1}{2} |\log (\Delta_n + 1) | < |\Delta_n | .
\end{equation}

We first handle the exceptional case when $\Lambda_n = 0$.  
Suppose that $\Lambda_n = 0$. Then $\Delta_n=0$, that is $b_1\al_1^n+b_2\al_2^n=0$.
Let
$$
K=\Q(u_0,\ldots,u_{m-1},\al_1,\ldots,\al_m).
$$
 Then, $[K:\Q] \le m!\cdot dD$. If $\al_1/\al_2$ is not a unit of $K$, then there exists a prime ideal $\p$ in the ring of integers of $K$ such that  $\ord_{\p} (\al_1/\al_2)$ is non-zero. Then, we get
 \begin{equation}
 \label{eq:ord1}
 n \le n | \ord_{\p} (\al_1/\al_2) | = |\ord_{\p} (b_2/b_1)|\le
 |\ord_{\p} (b_1)| + |\ord_{\p} (b_2)|.
 \end{equation}
 On the other hand, by definition, for any $1\le j \le m$, we know that
 $$
 |b_j|_\p=p^{-\ord_{\p}(b_j)/e_\p},
 $$
 where $p$ is the underlying prime number of $\p$, and $e_\p$ is the ramification index of $\p$ over $p$. 
 Noticing $b_j \in K$ and using \eqref{Liouville}, we obtain
 $$
 | \log |b_j|_{\p} | \le \frac{[K:\Q]}{d_\p}\h(b_j), 
 $$
 where $d_\p=[K_\p : \Q_p]$. Notice that $e_\p \le d_\p$. 
 So, for any $1\le j \le m$ we have
 \begin{equation}
 \label{eq:ord2}
| \ord_{\p}(b_j) | \le \frac{e_{\p}[K:\Q]}{d_\p \log p} \h(b_j) \le
 \frac{[K:\Q]}{ \log p} \h(b_j) \le 2B(u),
 \end{equation}
  where we use Lemma~\ref{lem:hbj} and $B(u)$ has been defined in Theorem~\ref{thm:dom}.
 Combining \eqref{eq:ord1} with \eqref{eq:ord2}, we get
 \begin{equation}
 \label{eq:n01}
  n\le 4B(u).
 \end{equation}

 Now, we suppose that $\Lambda_n = 0$ and $\al_1/\al_2$ is a unit of $K$.  
 Since $\al_1/\al_2$ is not a root of unity by assumption,
 there exists an embedding $\sigma: K \hookrightarrow \C$ such that $|\sigma(\al_1)/\sigma(\al_2)|>1$. 
 By \eqref{eq:udis1} and the definition of $J(f^*)$, we have
 $$
 |\sigma(\al_1)|-|\sigma(\al_2)| > J(f^*).
 $$
 On the other hand, let $G$ be the Galois group of the Galois closure of the field $\Q(a_0,a_1,\ldots, a_{m-1})$. 
 Then, each $\sigma(\al_j), j=1,2, \ldots, m$, is a root of the polynomial $\prod_{\tau \in G}\tau(f^*) \in \Z[X]$, 
 and so, by Lemma \ref{lem:Hg} and the definition of $I(f^*)$, we obtain 
 \begin{equation}
 \label{eq:sig-al}
 |\sigma(\al_j)| \le M\Big( \prod_{\tau \in G}\tau(f^*) \Big) \le I(f^*). 
 \end{equation}
 Notice that
$$
|\sigma(b_2)/\sigma(b_1)| \le \exp (2B(u)),
$$
which can be deduced similarly as \eqref{eq:|bj|}. In view of
$$
\sigma(b_1)\sigma(\al_1)^n+\sigma(b_2)\sigma(\al_2)^n=0,
$$
we deduce that
\begin{equation}
\label{eq:sigma}
|\sigma(\al_1)/\sigma(\al_2)|^n = |\sigma(b_2)/\sigma(b_1)| \le \exp (2B(u)).
\end{equation}

On the other hand, since
$$
|\sigma(\al_1)/\sigma(\al_2)|^n > \big( 1+J(f^*)/|\sigma(\al_2)| \big)^n
\ge (1+J(f^*)/I(f^*))^n,
$$
where the last inequality follows from \eqref{eq:sig-al}, we consider the inequality
$$
(1+J(f^*)/I(f^*))^n > \exp (2B(u)),
$$
which gives
$$
n > \frac{2B(u)}{\log (1+J(f^*)/I(f^*))}.
$$
Since $\log (1+x) > x -x^2/2 > x/2$ for $0< x <1$, it suffices to require that
\begin{equation}
\label{eq:n02}
n > 4B(u)I(f^*)J(f^*)^{-1}.
\end{equation}

Notice that the lower bound in \eqref{eq:n02} is much larger than the upper bound in \eqref{eq:n01}. Thus, if integer $n$ satisfies \eqref{eq:n02}, then the inequality in \eqref{eq:sigma} is not true, and we must have $\Lambda_n \ne 0$.

Now, we assume that $n$ satisfies \eqref{eq:n02}. So, $\Lambda_n \ne 0$. 
Applying Baker's inequality \eqref{Baker} to $\Lambda_n$, we find that
\begin{equation}
\label{eq:Lam}
|\Lambda_n | > \exp \big( -2^{38}D_1^2A_1A_2A_3 \log (eD_1) \log (en+2e) \big),
\end{equation}
where $D_1$ is the degree of the number field generated by $b_2/b_1$ and $\al_2/\al_1$ over $\Q$, and
\begin{align*}
& A_1 = \pi, \\
& A_2 \ge \max \{ D_1\h(b_2/b_1), |\log (b_2/b_1)|, 0.16 \} , \\
& A_3 \ge \max \{ D_1\h(\al_2/\al_1), |\log (\al_2/\al_1)|, 0.16 \}.
\end{align*}

 Since both $b_2/b_1$ and $\al_2/\al_1$ are contained in $K$, we have
$$
D_1 \le [K : \Q] \le m!\cdot dD.
$$

By Lemma~\ref{lem:hbj} and the definition of $B(u)$, we get
$$
D_1\h(b_2/b_1) \le D_1 (\h(b_1)+\h(b_2)) \le 2B(u).
$$
 In addition, by \eqref{Liouville} we note that
$$
| \log (b_2/b_1) | \le |\log | b_2/b_1 || + \pi \le D_1 \h(b_2/b_1)+ \pi \le  2B(u)+\pi.
$$
Thus, we choose 
$$
A_2 = 2B(u)+\pi. 
$$

Now, we want to choose $A_3$. 
Since $\al_1, \al_2$ are roots of the polynomial  $\prod_{\tau \in G} \tau(f^*) \in \Z[X]$, by \eqref{Mahler-Weil} and Lemma~\ref{lem:Hg} we have 
\begin{align*}
\h(\al_2/\al_1) \le \h(\al_2) + \h(\al_1)
 \le 2\log M\Big( \prod_{\tau \in G} \tau(f^*) \Big) \le 2\log I(f^*). 
\end{align*}
On the other hand, we have
$$
|\log (\al_2/\al_1) | \le |\log |\al_2/\al_1|| + \pi = \pi.
$$
So, we can choose
$$
A_3= 2\log I(f^*).
$$

Then, under \eqref{eq:n02} and recalling the definition of $C(u)$, the inequality \eqref{eq:Lam} becomes
\begin{equation*}
|\Lambda_n | > \exp(-C(u)\log n),
\end{equation*}
which, together with \eqref{eq:initial} and \eqref{eq:tran}, implies that
\begin{equation}
\label{eq:b1b2}
|b_1\al_1^n+b_2\al_2^n| > \frac{1}{2}|b_1\al_1^n| \exp(-C(u)\log n).
\end{equation}

Now, we are ready to find a lower bound for $n$ such that
$$
|b_1\al_1^n+b_2\al_2^n| > \sum_{j=3}^{m}|b_j\alpha_j^n|.
$$
This is implied in the following inequality by using \eqref{eq:b1b2}
$$
\frac{1}{2}|b_1\al_1^n| \exp(-C(u)\log n) \ge  \sum_{j=3}^{m}|b_j\alpha_j^n|.
$$
That is, we need that
$$
|b_1| \exp(-C(u)\log n) \ge 2\sum_{j=3}^{m}|b_j|\left(|\alpha_j|/|\al_1|\right)^n,
$$
which follows from the inequality
\begin{equation}
\label{eq:b12}
|b_1| \exp(-C(u)\log n) \ge 2\sum_{j=3}^{m}|b_j|\left(1-J(f^*)/|\al_1|\right)^n.
\end{equation}

By \eqref{eq:|bj|}, the inequality \eqref{eq:b12} is implied in the following inequality
$$
\exp(-B(u)-C(u)\log n) \ge 2m \exp(B(u)) \left(1-J(f^*)/|\al_1|\right)^n,
$$
which is equivalent to
$$
-n\log (1-J(f^*)/|\al_1|) - C(u)\log n \ge 2B(u)+ \log (2m).
$$
Since $-\log (1-x)=x+x^2/2+x^3/3+ \cdots$ for $0<x<1$ and noticing $|\al_1| \le M\Big( \prod_{\tau \in G}\tau(f^*) \Big) \le I(f^*)$, it suffices to require that
\begin{equation}
\label{eq:n1}
nJ(f^*)/I(f^*)-C(u)\log n \ge 2B(u)+ \log (2m).
\end{equation}

Notice that an integer $n$ satisfying the following inequalities
also satisfies \eqref{eq:n1},
\begin{equation}
\label{eq:n2}
\left\{\begin{array}{ll}
 C(u)\log n \le  nJ(f^*)/(2I(f^*)),\\
 \\
 nJ(f^*)/(2I(f^*)) \ge 2B(u)+ \log (2m).\\
\end{array}\right.
\end{equation}

Since the function $x/\log x$ is strictly increasing when $x\ge 3$, for $A\ge 3$, if $x\ge 2A\log A$, then $x/\log x \ge A$. 
Thus, if
\begin{equation}
\label{eq:n3}
n \ge 4C(u)I(f^*)J(f^*)^{-1} \log \big( 2C(u)I(f^*)J(f^*)^{-1}\big),
\end{equation}
then the first inequality in \eqref{eq:n2} holds, and in fact the second one also holds. 
Note that the lower bound in \eqref{eq:n3} is much bigger than that in \eqref{eq:n02}.

 So, if an integer $n$ satisfies the
inequality \eqref{eq:n3}, then we have 
$$
|b_1\al_1^n+b_2\al_2^n| > \sum_{j=3}^{m}|b_j\alpha_j^n|.
$$
Thus, $u_n\ne 0$. This completes the proof of the theorem.

\section*{Acknowledgements}

The author wants to thank Manas Patra and Igor Shparlinski for introducing him into the Skolem Problem and also for useful discussions and helpful comments.  
The research was supported by the Australian Research Council Grant DP130100237 and also by the Macquarie University Research Fellowship.


\begin{thebibliography}{1}

%\bibitem{Smith1950}
%J.~Smith.
%\newblock A conjecture on sequences with unusual properties.
%\newblock {\em J. Major Results}, 5(2):100--200, 1950.


%\bibitem{BDJB}
%P.C. Bell, J.-C. Delvenne, R.M. Jungers and V.D. Blondel, \textit{The continuous Skolem-Pisot problem}, Theor. Comput. Sci. \textbf{411} (2010),  3625--3634.

%\bibitem{Bell}
%J.P. Bell and S. Gerhold, \emph{On the positivity set of a linear recurrence sequence}, Israel J. Math. \textbf{157}(1) (2007), 333--345.

%\bibitem{BBGMS}
%C.D. Bennett, J. Blass, A.M.W. Glass, D.B. Meronk and R.P. Steiner, 
%\textit{Linear forms in the logarithms of three positive rational numbers}, J. Th{\'e}or. Nombres Bordeaux, \textbf{9} (1997), 97--136.

\bibitem{BM}
J. Berstel and M. Mignotte, \textit{Deux propri\'et\'es d\'ecidables des suites r\'ecurrentes lin\'eaires}, Bull. Soc. Math. France, \textbf{104} (1976), 175--184.

\bibitem{BP}
V.D. Blondel and N. Portier, \textit{The presence of a zero in an integer linear recurrent sequence is NP-hard to decide}, Linear Algebra Appl., \textbf{351-352} (2002), 91--98.

\bibitem{BDPS}
Y. Bugeaud, A. Dujella, T. Pejkovi{\' c} and B. Salvy, \textit{Absolute real root separation}, Amer. Math. Monthly, \textbf{124} (2017),  930--936. 

%\bibitem{Burke}
%J.R. Burke and W.A. Webb, \emph{Asymptotic behaviour of linear recurrences}, Fibonacci Quart. \textbf{19}(4) (1981),  318--321.

%\bibitem{Chela}
%R.~Chela, \emph{Reducible polynomials}, J. Lond. Math. Soc., {\bf 38} (1963), 183--188.

\bibitem{Chonev}
V. Chonev, \textit{Reachability problems for linear dynamical systems}, PhD thesis, University of Oxford, 2015. 

\bibitem{COW}
V. Chonev, J. Ouaknine and J. Worrell, \textit{On the complexity of the orbit problem}, J. ACM, \textbf{63}(3):23:1--23:18, 2016.

\bibitem{DS1}
A. Dubickas and M. Sha, \textit{Counting degenerate polynomials of fixed degree and bounded height}, 
Monatsh. Math., \textbf{177} (2015), 517--537.  

\bibitem{DS2}
A. Dubickas and M. Sha, \textit{Counting and testing dominant polynomials}, Exper. Math.,  \textbf{24} (2015), 312--325.

\bibitem{DS3}
A. Dubickas and M. Sha, \textit{Positive density of integer polynomials with some prescribed properties}, J. Number Theory, \textbf{159} (2016), 27--44.

\bibitem{DSS}
A. Dubickas, M. Sha and I. Shparlinski, \textit{Explicit form of Cassels' $p$-adic embedding theorem for number fields}, 
Can. J. Math., \textbf{67} (2015), 1046--1064.

\bibitem{EPSW}
G. Everest, A. van der Poorten, I. Shparlinski and T. Ward, {\it Recurrence Sequences},
Amer. Math. Soc., Providence, RI, 2003.

%\bibitem{Gerhold}
%S. Gerhold, \emph{Point lattices and oscillating recurrence sequences}, J. Differ. Equ. Appl. \textbf{11}(6) (2005), 515--533.

\bibitem{HHHK}
V. Halava, T. Harju, M. Hirvensalo and J. Karhum\"aki, \textit{Skolem's problem -- on the border between decidability and undecidability}, Technical  Report 683, Turku Centre for Computer Science, 2005.

\bibitem{Klinger1967}
A. Klinger, \emph{The Vandermonde matrix}, Amer. Math. Monthly, \textbf{74} (1967), 571--574.

%\bibitem{Kuba}
%G.~Kuba, \emph{On the distribution of reducible polynomials}, Math. Slovaca {\bf 59} (2009), 349--356.

\bibitem{Litow}
B. Litow, \textit{A decision method for the rational sequence problem}, In: Electronic Colloquium on Computational Complexity (ECCC), volume 4, 1997. 

\bibitem{Mahler} 
K. Mahler, \textit{An inequality for the discriminant of a polynomial}, Michigan Math. J,. \textbf{11} (1964), 257--262. 

\bibitem{Matveev}
E.M. Matveev, \textit{An explicit lower bound for a homogeneous rational linear form in the logarithms of algebraic numbers II}, Izv. Math., \textbf{64}(6) (2000), 1217--1269.

%\bibitem{MG}
%M. Mignotte and P. Glesser, \emph{Landau's Inequality via Hadamard's}, J. Symb. Comput., \textbf{18}(4) (1994), 379--383.

\bibitem{Mig}
M. Mignotte, \textit{A kit on linear forms in three logarithms}, available at \url{http://irma.math.unistra.fr/~bugeaud/travaux/kit.pdf}.

\bibitem{Mignotte}
M.~Mignotte and D.~\c{S}tef\u{a}nescu, \emph{Polynomials: an algorithmic approach}, Springer, Singapore, 1999.


\bibitem{MST}
M. Mignotte, T.N. Shorey and R. Tijdeman, \textit{The distance between terms of an algebraic recurrence sequence}, J. Reine Angew. Math., \textbf{349} (1984), 63--76.

%\bibitem{Mishra}
%B. Mishra, \textit{Algorithmic algebra}, Springer, New York, 1993.

\bibitem{OW12}
J. Ouaknine and J. Worrell, \textit{Decision problems for linear recurrence sequences}, In:
Proc. 6th Intern. Workshop on Reachability Problems (RP), LNCS \textbf{7550}, pp. 21--28, Springer, 2012.

%\bibitem{OW14a}
% J. Ouaknine and J. Worrell, \textit{Positivity problems for low-order linear recurrence sequences}, In:  Proc. 25th Symp. on Discrete Algorithms (SODA), pp. 366--379, ACM-SIAM, 2014. %Full version as arXiv:1307.2779.

%\bibitem{OW14b}
%J. Ouaknine and J. Worrell, \textit{On the Positivity Problem for simple linear recurrence sequences},
%In: Proc. 41st Intern. Colloq. on Automata, Languages and Programming (ICALP), LNCS \textbf{8573}, pp. 318--329, Springer, 2014. %Full version as  arXiv:1309.1550.

%\bibitem{OW14c}
%J. Ouaknine and J. Worrell, \textit{Ultimate Positivity is decidable for simple linear recurrence sequences}, In: Proc. 41st Intern. Colloq. on Automata, Languages and Programming (ICALP), LNCS \textbf{8573}, pp. 330--341, Springer, 2014. %Full version as arXiv:1309.1914.

\bibitem{Tao}
T. Tao, \textit{Structure and randomness: pages from year one of a mathematical blog}, 
Amer. Math. Soc., Providence, RI, 2008.


\bibitem{Ver}
N.K. Vereshchagin, \textit{Occurrence of zero in a linear recursive sequence}, Math. Notes, \textbf{38} (1985), 609--615.

\bibitem{Waldschmidt}
M. Waldschmidt, \emph{Diophantine approximation on linear algebraic groups. Transcendence properties of the exponential function in several variables,} Grundlehren der Mathematischen Wissenschaften \textbf{326}, Springer, Berlin, 2000.

\end{thebibliography}
\end{document}